\documentclass[11pt]{article}
\usepackage{fullpage}
\usepackage[hyperindex,breaklinks,colorlinks,citecolor=blue]{hyperref}\usepackage{amsmath, amssymb, amsthm, graphicx}
\usepackage{subcaption}
\usepackage[shortlabels]{enumitem}

\title{Computability of finite simplicial complexes}
\author{Djamel Eddine Amir and Mathieu Hoyrup\\ \\
{\small\it Universit\'e de Lorraine, CNRS, Inria, LORIA, F-54000 Nancy, France}\\
{\small\it djamel-eddine.amir@loria.fr, mathieu.hoyrup@inria.fr}} 

\date{}

\theoremstyle{plain}
\newtheorem{corollary}{Corollary}[section]
\newtheorem{proposition}{Proposition}[section]
\newtheorem{lemma}{Lemma}[section]
\newtheorem{theorem}{Theorem}[section]
\newtheorem*{theorem*}{Theorem}

\theoremstyle{definition}
\newtheorem{definition}{Definition}[section]
\newtheorem*{definition*}{Definition}

\theoremstyle{remark}
\newtheorem{remark}{Remark}[section]
\newtheorem{fact}{Fact}[section]
\newtheorem{example}{Example}[section]
\newtheorem{notation}{Notation}[section]
\newtheorem{question}{Question}
\newtheorem{claim}{Claim}

\newenvironment{claimproof}{\begin{proof}[Proof of the claim]}{\end{proof}}

\newcommand{\interior}[1]{\mathrm{int}(#1)}
\renewcommand{\int}[2]{\mathrm{int}_{#1}{#2}}
\newcommand{\N}{\mathbb{N}}
\newcommand{\cone}[1]{\mathrm{Cone}(#1)}

\newcommand{\supp}{\mathrm{supp}}
\newcommand{\id}{\mathrm{id}}
\newcommand{\norm}[1]{\lVert#1\rVert}

\newcommand{\B}{\mathbb{B}}
\newcommand{\Q}{\mathbb{Q}}
\newcommand{\R}{\mathbb{R}}
\renewcommand{\S}{\mathbb{S}}
\newcommand{\Ne}{\mathcal{N}}

\newcommand{\oddbd}[1]{\partial_\mathrm{odd}#1}

\newcommand{\cB}{\overline{B}}

\begin{document}
	\maketitle

\begin{abstract}
The topological properties of a set have a strong impact on its computability properties. A striking illustration of this idea is given by spheres and closed manifolds: if a set~$X$ is homeomorphic to a sphere or a closed manifold, then any algorithm that semicomputes~$X$ in some sense can be converted into an algorithm that fully computes~$X$. In other words, the topological properties of~$X$ enable one to derive full information about~$X$ from partial information about~$X$. In that case, we say that~$X$ has computable type. Those results have been obtained by Miller, Iljazovi{\'c}, Su{\v s}i{\'c} and others in the recent years. A similar notion of computable type was also defined for pairs~$(X,A)$ in order to cover more spaces, such as compact manifolds with boundary and finite graphs with endpoints.

We investigate the higher dimensional analog of graphs, namely the pairs $(X,A)$ where~$X$ is a finite simplicial complex and~$A$ is a subcomplex of~$X$. We give two topological characterizations of the pairs having computable type. The first one uses a global property of the pair, that we call the~$\epsilon$-surjection property. The second one uses a local property of neighborhoods of vertices, called the surjection property. We give a further characterization for~$2$-dimensional simplicial complexes, by identifying which local neighborhoods have the surjection property. 

Using these characterizations, we give non-trivial applications to two famous sets: we prove that the dunce hat does not have computable type whereas Bing's house does. Important concepts from topology, such as absolute neighborhood retracts and topological cones, play a key role in our proofs.
\end{abstract}
%

\section{Introduction}
Computable analysis is a theory formalizing computations on real numbers using finite but arbitrary precision, and allowing to investigate the theoretical possibility of solving problems on real numbers. The computable aspects of topology are an important research topic in computable analysis. Computability of homology groups was investigated in \cite{Collins09}, computability of planar continua in \cite{Kihara12}, computability of the Brouwer fixed-point theorem was studied in \cite{Neumann18} and \cite{BrattkaRMP19}, and computability of Polish spaces is addressed in \cite{Harrison-Trainor20}.

A particularly rich topic is the computability of subsets of the plane and of Euclidean spaces. For instance, the computability of Julia sets has thoroughly been studied \cite{Braverman2008}, the computability of the Mandelbrot set is still an open problem \cite{Hertling05} and the computability of the set of solutions of a computable equation is generally a non-trivial problem \cite{LeRoux15}.

These studies reveal that many natural definitions of sets induce a \emph{semi-algorithm}, and finding a proper \emph{algorithm} computing the set can be challenging. Informally, a compact subset of the plane is \emph{semicomputable} if there is an algorithm that for each pixel, semidecides whether the pixel is disjoint from the set, i.e.~halts exactly in that case. A compact subset of the plane is  \emph{computable} if there is an algorithm that decides, for each pixel, whether it intersects the set. This idea can be generalized to higher dimensions, and to subsets of many mathematical spaces.

Although semicomputability of compact sets is strictly weaker than computability in general, it turns out that they are equivalent for many natural sets, and that this phenomenon comes from the topological properties of these sets. For instance, it was prove  by Miller \cite{2002MILLER} that semicomputability and computability are equivalent for spheres, and for every set that is homeomorphic to a sphere. This result leads to the following definition: say that a compact space~$X$ has \textbf{computable type} if any semicomputable set~$Y$ that is homeomorphic to~$X$ is actually computable. This property has been intensively studied by Miller \cite{2002MILLER} and more recently by Iljazovi{\'c} and its co-authors \cite{BurnikI14,2018manifolds,CickovicIV19,2020graphs,celar21,CelarI21} in the recent years. A striking aspect of this property is that it builds a bridge between computability theory and topology. The following results were obtained:
\begin{itemize}
\item The~$n$-dimensional sphere~$\S_n$ (which is the higher dimensional analog of the circle) has computable type \cite{2002MILLER},
\item Every closed~$n$-manifold (these are compact spaces which are locally homeomorphic to~$\R^{n}$, for instance the~$n$-dimensional sphere and the~$n$-dimensional torus)  has computable type \cite{2018manifolds}.
\end{itemize}
A line segment or a disk fails to have this property: it is not difficult to build a semicomputable disk which is not computable. However, a similar result can be proved if one requires in addition that the boundary of the set is semicomputable. It leads to the following generalization from compact spaces~$X$ to pairs~$(X,A)$ where~$X$ and~$A\subseteq X$ are compact: a pair~$(X,A)$ has \textbf{computable type} if for any semicomputable pair~$(Y,B)$ that is homeomorphic to~$(X,A)$,~$Y$ is computable. The following results have been obtained for pairs:
\begin{itemize}
\item The~$n$-dimensional ball (which is the higher dimensional analog of the disk) with its bounding sphere~$(\B_n,\S_{n-1})$ has computable type \cite{2002MILLER},
\item Every compact manifold with boundary~$(M,\partial M)$ has computable type \cite{2018manifolds},
\item Every finite (topological) graph~$(G,V_1)$, where~$V_1$ is the set of vertices of degree~$1$, has computable type \cite{2020graphs}.
\end{itemize}

Our goal in this paper is to study the property of having computable type for a broader class of spaces, to characterize the pairs having computable type and to develop a unifying argument for the known examples. Our first observation is that graphs and manifolds have the common property that they are locally topological cones as follows (see Figure \ref{fig_cones_intro} for an illustration of this idea):
\begin{itemize}
\item A finite graph is locally a cone of a finite set,
\item A $2$-dimensional manifold is locally a disk, which is the cone of a circle, and more generally an~$n$-dimensional manifold is locally an~$n$-ball, which is the cone of an~$(n-1)$-sphere.
\end{itemize}

In this article, we study the class of finite simplicial complexes which is a large class of spaces that are also locally topological cones, as illustrated in Figure \ref{fig_cones_simplicial}. 

\begin{figure}[h]
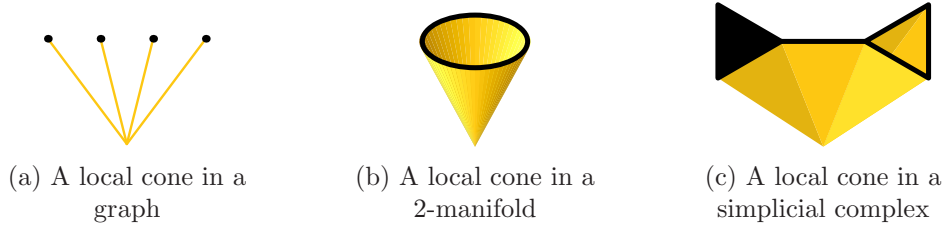

\centering
\subcaptionbox{A local cone in a graph}[3.5cm]{\includegraphics{image-20}}\hspace{1cm}
\subcaptionbox{A local cone in a $2$-manifold}[3.5cm]{\includegraphics{image-21}}\hspace{1cm}
\subcaptionbox{A local cone in a simplicial complex\label{fig_cones_simplicial}}[3.5cm]{\includegraphics{image-23}}
\caption{Examples of local cones in 3 types of spaces.}\label{fig_cones_intro}
\end{figure}

Finite simplicial complexes are the higher dimensional analogs of finite graphs. They are made of simplices that are attached together along their faces. This class of compact topological spaces is large enough to include many examples (e.g., most common compact manifolds, geometrical models from computer graphics) and can be easily described using finite combinatorial information, so we can hope to obtain a full characterization of computable type for them. We do not consider infinite simplicial complexes because the usual topologies make them non-compact.
	
Let~$ (X,A) $ be a pair consisting of a finite simplicial complex~$X$ and a subcomplex~$A$. We call such a pair a \textbf{simplicial pair}. Our main problem is to understand which simplicial pairs~$(X,A)$ have computable type. We give a thorough answer, by giving two topological characterizations of the simplicial pairs~$(X,A)$ having computable type. One of them is global whereas the other one is local. The local characterization makes it very easy to check whether a simplicial pair~$(X,A)$ has computable type, by inspecting the neighborhoods of each vertex separately. Those neighborhoods are called \textbf{local cones}, because they are topological cones with the vertex as the tip (precise definitions will be given in the article). We then use the local characterization to prove or disprove that specific sets, such as \textbf{Bing's house} and the \textbf{dunce hat}, have computable type. The previous techniques developed in the literature were too specific to be applied to these sets. Our techniques not only make it possible to treat any simplicial complex, but also provide a simple and visual way to settle the question for many sets. 

The proofs are non-trivial but the statements are elegant and easy to apply. For instance it is very easy to apply our results to show that the dunce hat does not have computable type. However, the internals of the proofs of the theorems are rather involved and we are not aware of any simpler, more direct argument. Therefore our results provide significant progress in the understanding of the computable type property. Moreover, our approach in this article is new in the sense that the proofs are very different from the arguments developed in the literature on the computable type property.

It turns out that the computability property we are studying is intimately related to topology, so we need to use topology in our investigation. However, we only assume familiarity with basic topology (e.g., continuity and compactness). When we use more advanced topological notions, we give the necessary background (e.g.~cones, simplicial complexes).

\subparagraph{The results.}
Let us summarize the main results of this paper. We will be working with pairs~$(X,A)$ consisting of a compact metric space~$X$ and a compact subset~$A$, to be informally thought as the boundary of~$X$. A typical example is given by the pair~$(\B_{n+1},\S_n)$ consisting of the~$(n+1)$-dimensional ball and the~$n$-dimensional sphere:
\begin{align*}
\B_{n+1}&=\{x\in \R^{n+1}:\norm{x}\leq 1\},\\
\S_n&=\{x\in\R^{n+1}:\norm{x}=1\},
\end{align*}
where~$\norm{\cdot}$ is the Euclidean norm or any equivalent norm. We introduce two important properties of pairs, given in Definition \ref{def_surjection_property} and restated here.

\begin{definition*}
A pair~$(X,A)$ has the \textbf{surjection property} if every continuous function~$f:X\to X$ satisfying~$f|_A=\id_A$ is surjective.

Let~$\epsilon>0$. A pair~$(X,A)$ has the~$\epsilon$\textbf{-surjection property} if every continuous function~$f:X\to X$ satisfying~$f|_A=\id_A$ and~$d(f,\id_X)<\epsilon$ is surjective.
\end{definition*}

For instance, a consequence of Brouwer's fixed-point theorem is that 
the pair~$(\B_{n+1},\S_n)$ has the surjection property.

The main result of the paper is Theorem \ref{thm_main}, which relates computable type with these two properties. We restate it here. We recall that a \textbf{simplicial pair}~$(X,A)$ consists of a finite simplicial complex~$X$ and a subcomplex~$A\subseteq X$.

\begin{theorem*}
Let~$(X,A)$ be a simplicial pair such that~$A$ has empty interior in~$X$. The following conditions are equivalent:
\begin{enumerate}
\item $(X,A)$ has computable type,
\item There exists~$\epsilon>0$ such that~$(X,A)$ has the~$\epsilon$-surjection property,
\item Every local cone pair of~$(X,A)$ has the surjection property.
\end{enumerate}
\end{theorem*}

Condition 2.~is the global property mentioned above and condition 3.~is the local one. This theorem reduces a computability-theoretic property to purely topological ones. We develop further techniques to determine whether a pair has computable type, by applying this theorem or by analyzing when the topological properties are satisfied. The first one is stability under finite unions (Theorem \ref{thm_finiteunion} and Corollary \ref{cor_finiteunion}).

\begin{theorem*}[Finite union]
Let~$(X,A)$ be a simplicial pair and~$(X_i,A_i)_{i\leq n}$ be pairs of subcomplexes such that~$X=\bigcup_{i\leq n}X_i$ and~$A=\bigcup_{i\leq n}A_i$. If each~$(X_i,A_i)$ has computable type, then~$(X,A)$ has computable type.
\end{theorem*}

The second one is a further characterization of the~$2$-dimensional simplicial pairs having computable type, by reduction the surjection property for local cone pairs to a simple property of graphs (Theorem \ref{thm_cone_graph}). We demonstrate the strength of that result by giving non-trivial applications to two famous sets: the dunce hat (Figure \ref{fig_dunce_a}) and Bing's house (Figure \ref{fig_bing}).

In order to make the paper understandable to a larger audience, we give informal proofs of the main results. The detailed proofs are then given in the appendix.

The paper is organized as follows. In Section \ref{sec:Preliminaries},
we give the needed background on computability of sets, simplicial complexes and cone spaces. In Section \ref{sec:The-surjection-property},
we define the surjection property and the~$\epsilon$-surjection
property, state and prove our main result. In Section \ref{sec:Finite-unions-and}, we present techniques to prove or disprove the ($\epsilon$-)surjection property. As an application, we prove that the dunce hat does not have computable type whereas the Bing's house does,  by studying the local cones of each of the two sets. In Section \ref{sec_boundary}, we briefly discuss the possible notions of boundary~$\partial X$ of a simplicial complex~$X$ that make the pair~$(X,\partial X)$ have computably type. We finally formulate open questions and discuss a generalization of our results in Section \ref{sec_conclusion}. As previously mentioned, the reader can find the detailed proofs of all the results in the appendix.

\section{\label{sec:Preliminaries}Preliminaries}

We give here some necessary preliminaries in computability theory
and topology. We start with this central definition.
\begin{definition}
A \textbf{pair}~$(X,A)$ consists of a compact metrizable space
$X$ and a compact subset~$A\subseteq X$. A \textbf{copy} of a
pair~$(X,A)$ in a topological space~$Z$ is a pair~$(Y,B)$ such that~$Y\subseteq Z$
is homeomorphic to~$X$ and~$A$ is sent to~$B$ by the homeomorphism. 
\end{definition}


\subsection{Computability of sets}

We recall definitions and results about the Hilbert cube and computable type. We will mainly use the following notion from computability theory: a set~$A\subseteq\N$ is \textbf{computably enumerable (c.e.)} if there exists a Turing machine that, on input~$n\in\N$, halts if and only if~$n\in A$. This notion immediately extends to subsets of countable sets, whose elements can be encoded by natural numbers. 

\subparagraph{Computability in the Hilbert cube.}
We work in the Hilbert cube because it is universal among the separable metrizable spaces, in particular every compact metrizable space embeds in the Hilbert cube.
\begin{definition}
The \textbf{Hilbert cube} is the space~$Q=[0,1]^\N$ endowed with the metric~$d(x,y)=\sum_i2^{-i}|x_i-y_i|$.  We let~$(B_i)_{i\in\N}$ be a computable enumeration of the open balls~$B(x,r)$ where~$x\in Q$ has finitely many non-zero rational coordinates and~$r>0$ is rational; these $B_{i}$s are called \textbf{rational balls}.
\end{definition}
\begin{notation}
If~$X\subseteq Q$ and~$f,g:X\to Q$, then let
\begin{equation*}
d_X(f,g)=\sup_{x\in X}d(f(x),g(x)).
\end{equation*}
\end{notation}

We recall definitions of computability of compact subsets of the Hilbert cube. The reader can find more details about computability of sets in \cite{BRATTKA200343,IK20}.

\begin{definition}[Computability of sets]
A compact set~$X\subseteq Q$ is:
\begin{itemize}
\item \textbf{Semicomputable} if there exists a c.e.~set~$E\subseteq\N$ such that~$Q\setminus X=\bigcup_{i\in E}B_i$,
\item \textbf{Computable} if it is semicomputable and~$\left\{ i\in\N:X\cap B_{i}\neq\emptyset\right\}$ is c.e.
\end{itemize}
A pair~$(X,A)$ in~$Q$ is \textbf{semicomputable} if both~$X$ and~$A$ are
semicomputable.
\end{definition}

Intuitively,~$X$ is semicomputable if there is an algorithm that takes rational cube as input (a voxel) and semidecides whether that cube is disjoint from~$X$, i.e.~halts exactly in this case. $X$ is computable if there is an algorithm that \emph{decides} whether a cube intersects the set.

For instance, the Mandelbrot set is semicomputable because its definition gives an algorithm that can eventually detect that a point is outside this set; whether it is computable is an open problem, see \cite{Hertling05}.

\begin{example}
The line segment~$I=[0,1]$ embedded in the simplest way as~$[0,1]\times Q\subseteq Q$ is computable. However, if~$A\subseteq \N$ is the halting set (a non-computable c.e.~set) and~$x_A=\sum_{n\in A}2^{-n}$, then~$[x_A,1]\times Q$ is a copy of~$I$ which is semicomputable but not computable.
\end{example}

The Hilbert cube itself is a computable subset of itself. A compact set~$X\subseteq Q$ is semicomputable if and only if the set
\begin{equation*}
\left\{ (i_{1},\ldots,i_{n})\in\N^{\ast}:X\subseteq B_{i_{1}}\cup\ldots\cup B_{i_{n}}\right\}
\end{equation*}
is c.e., and it is computable if and only if in addition it contains a dense computable sequence. A function~$f:Q\to Q$ is \textbf{computable} if there exists a c.e.~set~$E\subseteq\N^2$ such that~$f^{-1}(B_i)=\bigcup_{(i,j)\in E} B_j$. The image of a (semi)computable set under a computable function is a (semi)computable set. Semicomputable sets have very useful properties: if~$X\subseteq Q$ is semicomputable and~$f,g:X\to Q$ are computable, then~$\{q\in\Q:d_X(f,g)<q\}$ is c.e.

\subparagraph{Computable type.}

The next definition is the main notion of this article (see \cite{2018manifolds}).
\begin{definition}
A pair~$(X,A)$ has \textbf{computable type} if for every semicomputable copy~$(Y,B)$ of the pair in the Hilbert cube,~$Y$ is computable.

A compact space~$X$ has computable type if the pair~$(X,\emptyset)$
has.
\end{definition}

\begin{remark}\label{rmk_hilbert}
In fact, in \cite{2018manifolds} computable type was defined separately for
copies in computable metric spaces and computably Haudorff spaces. In a forthcoming article, we show that taking
the copies in computably Hausdorff spaces, computable metric spaces
or the Hilbert cube are all equivalent using the fact that computable metric spaces embed effectively in the Hilbert cube, as well as 
Schr\"oder's computable metrization theorem \cite{1998Schroder}.
\end{remark}

\subsection{Topology}

We recall some notions which will be used, like simplicial complexes and cone spaces. We will work with compact metrizable spaces only, and may omit this assumption in the statements.
\begin{definition}
Let~$(X,A)$ be a pair. A \textbf{retraction}~$r:X\rightarrow A$
is a continuous function such that~$r|_{A}=\id_{A}$. If a retraction exists, then
we say that~$A$ is a \textbf{retract} of~$X$.
\end{definition}

\subparagraph{Simplicial complex.}
Let~$V=\{0,\ldots,n\}$ and~$P_+(V)$ be the set of non-empty subsets of ~$V$. An \textbf{abstract
finite simplicial complex} is a set~$S\subseteq P_+(V)$
such that if~$\sigma\in S$ and~$\emptyset\neq\sigma^{\prime}\subset\sigma$,
then~$\sigma^{\prime}\in S$. Its elements~$\sigma\in S$ are called the \textbf{simplices} of~$S$. If~$\sigma\in S$ has~$n+1$ elements, then~$\sigma$ is an \textbf{$n$-simplex}. The \textbf{vertices} of~$S$ are the singletons~$\{i\}\in S$.   $\sigma\in S$ is \textbf{free} if there exists exactly one~$\sigma'\in S$ with~$\sigma\subsetneq \sigma'$. A \textbf{subcomplex} of~$S$ is an abstract simplicial complex contained in~$S$. 

The \textbf{support} of a vector~$x=(x_{0},\ldots,x_{n})\in [0,1]^{n+1}$
is~$\supp(x)=\left\{ i:x_{i}\neq0\right\}~$. The \textbf{standard
realization} of an abstract simplicial complex~$S$ is the set
\begin{equation*}
|S|=\Big\{x=(x_{0},\ldots,x_{n})\in[0,1]^{n+1}:\sum_{i}x_{i}=1,\supp(x)\in S\Big\}.
\end{equation*}
Any space homeomorphic to the standard realization of an abstract finite simplicial complex is called a \textbf{finite simplicial complex}. We often identify an abstract simplicial complex and its standard realization.


%

A \textbf{simplicial pair}~$(X,A)$ consists of a finite simplicial complex~$X$ and a subcomplex~$A$.

\begin{remark}
For technical reasons, we will implicitly assume that~$A$ contains all the free vertices of~$X$ (those are the points~$x$ having a neighborhood homeomorphic to~$[0,1)$ with a homeomorphism sending~$x$ to~$0$).
\end{remark}
%
%

In a simplicial complex, each vertex has a neighborhood which is usually called a star and is topologically a cone. Our main result will relate the computable type property with a property of these local cones. Because we are dealing with pairs, we need to define local cone pairs, as follows. 
\begin{definition}
Let~$(X,A)$ be the standard realization of a simplicial pair and~$v_i=(0,\ldots,1,\ldots,1)$ be a vertex. The \textbf{local cone pair} at~$v_i$ is~$(K_i,M_i)$ defined by:
\begin{align*}
K_i&=\{x\in X:x_i\geq 1/2\},\\
M_i&=\{x\in X:x_i=1/2\}\cup (K_i\cap A).
\end{align*}
\end{definition}

Note that the coefficient~$1/2$ is arbitrary and could be replaced by any number in~$(0,1)$. 

\begin{remark}\label{rmk_conepair}
We call~$(K_i,M_i)$ a cone pair because~$K$ is a topological cone: let~$L_i=\{x\in X:x_i=1/2\}$,~$K_i$ is a copy of the cone of~$L_i$, obtained from~$L_i\times [0,1]$ by identifying all the points~$(l,0)$ together. The point obtained by this identification is the tip of the cone and corresponds to the vertex~$v_i$. If~$N_i=\{x\in A:x_i=1/2\}$, then~$M_i$ is the union of~$L_i$ and of the cone of~$N_i$.

In the language of simplicial complexes,~$K_i$ corresponds to the \emph{star} of~$v_i$ and~$L_i$ to the \emph{link} of~$v_i$. $K_i$ is homeomorphic to the union of simplices containing~$v_i$. Each such simplex has a face that does not contain~$v_i$, and~$L_i$ is the union of these faces.
\end{remark}


\section{The (\texorpdfstring{$\epsilon$}{epsilon}-)surjection property and computable type for simplicial pairs}\label{sec:The-surjection-property}

We now present the main result of this paper, that identifies which simplicial pairs have computable type, using the following topological properties.
\begin{definition}\label{def_surjection_property}
A pair~$(X,A)$ has the \textbf{surjection property} if every continuous
function~$f:X\rightarrow X$ satisfying~$f|_{A}=\id_{A}$ is surjective.

A pair~$(X,A)$ in~$Q$ has the \textbf{$\epsilon$-surjection property}
for some~$\epsilon>0$, if every continuous function~$f:X\rightarrow X$
satisfying~$f|_{A}=\id_{A}$ and~$d_{X}(f,\id_{X})<\epsilon$ is surjective.
\end{definition}

\begin{example}
\label{example ball sphere}
\begin{itemize}
\item For every~$n\in\N$, the~$(n+1)$-dimensional
ball and its bounding~$n$-dimensional sphere form a pair~$(\mathbb{B}_{n+1},\mathbb{S}_{n})$ that has the surjection property. It is a consequence of an equivalent formulation of Brouwer's fixed-point theorem that~$\S_n$ is not a retract of~$\B_{n+1}$ (Corollary 2.15 in \cite{Hatcher2002}).
\item The pair~$(\S_n,\emptyset)$ does not have the surjection property (take a constant function~$f:\S_n\to\S_n$), but has the~$\epsilon$-surjection property if~$\epsilon$ is sufficiently small. It can be proved using classical results in topology, or as a consequence of Theorem \ref{thm_main} below.
\end{itemize}
\end{example}

Although the~$\epsilon$-surjection property depends on the particular copy of a pair~$(X,A)$, quantifying over~$\epsilon$ yields a topological invariant, i.e.~a property of the pair that is satisfied either by all copies or by none of them.
\begin{proposition}
Whether there exists~$\epsilon>0$ such that~$(X,A)$ has the~$\epsilon$-surjection property does not depend on the copy of~$(X,A)$ in~$Q$.
\end{proposition}
\begin{proof}
If~$(Y,B)$ is a copy of~$(X,A)$, then let~$\phi:X\to Y$ be a homeomorphism such that~$\phi(A)=B$. By compactness of~$X$,~$\phi$ is uniformly continuous so given~$\epsilon>0$, there exists~$\delta>0$ such  that if~$d(x,x')<\delta$ then~$d(\phi(x),\phi(x'))<\epsilon$. If~$(Y,B)$ has the~$\epsilon$-surjection property, then we show that~$(X,A)$ has the~$\delta$-surjection property. Let~$f:X\to X$ be continuous, satisfying~$f|_A=\id_A$ and~$d_X(f,\id_X)<\delta$. Define~$g=\phi\circ f\circ \phi^{-1}:Y\to Y$: one has~$g|_B=\id_B$ and~$d_Y(g,\id_Y)<\epsilon$ by choice of~$\delta$ so~$g$ is surjective, hence~$f$ is surjective.
\end{proof}

We now state the main result of this paper.
\begin{theorem}[The main theorem]\label{thm_main}Let~$(X,A)$ be a simplicial pair such that~$A$ has empty interior in~$X$. The following statements are equivalent:
\begin{enumerate}
\item~$(X,A)$ has computable type,
\item~$(X,A)$ has the~$\epsilon$-surjection property for some~$\epsilon>0$,
\item All the local cone pairs~$(K_i,M_i)$ have the surjection property.
\end{enumerate}
\end{theorem}

We separate the proof into several independent parts.

\begin{remark}
A single topological space~$X$ has many different simplicial decompositions, i.e.~many abstract simplicial complexes whose realizations are homeomorphic to~$X$. For instance, a triangle can be decomposed into many smaller triangles. At first sight, the third condition in Theorem \ref{thm_main} depends on the choice of the decomposition, because the local cone pairs are taken at the vertices of the decomposition. However, the theorem implies that the choice of the simplicial decomposition is irrelevant, because conditions 1.~and 2.~do not depend on the decomposition: if all the cone pairs in a simplicial decomposition have the surjection property, then it is still true for all other simplicial decompositions of the space.
\end{remark}

For a simplicial pair that is itself homeomorphic to a cone pair, we obtain a further equivalence, which is a consequence of Theorem \ref{thm_main}.
\begin{corollary}\label{cor_cone_pair}
Let~$(X,A)$ be a simplicial cone pair such that~$A$ has empty interior in~$X$. The following statements are equivalent:
\begin{enumerate}
\item $(X,A)$ has computable type,
\item $(X,A)$ has the~$\epsilon$-surjection property for some~$\epsilon>0$,
\item $(X,A)$ has the surjection property.
\end{enumerate}
\end{corollary}
\begin{proof}
The surjection property implies the~$\epsilon$-surjection property for any pair. Conversely, if the pair~$(X,A)$ has the~$\epsilon$-surjection property then each local cone pair has the surjection property, but~$(X,A)$ is itself a local cone pair.
\end{proof}

The rest of this section is devoted to the proof of this result. We will give several applications in the next section.
 
\subsection{The \texorpdfstring{$\epsilon$}{epsilon}-surjection property implies computable type}
In this section we give an informal idea of the proof of~$2.\Rightarrow 1.$ in Theorem \ref{thm_main}. The details can be found in the appendix, Section \ref{sec_proof21}. The idea of the proof is that if~$A$ has empty interior in~$X$ and~$(X,A)$ has the~$\epsilon$-surjection property, then for an open set~$U$ the following conditions are equivalent:
\begin{itemize}
\item $U$ intersects~$X$,
\item There exists a continuous non-surjective function~$g:(X\setminus U)\cup A\to X$ such that~$g|_A=\id_A$ and~$d_X(g,\id_X)<\epsilon$.
\end{itemize}
This equivalence is straightforward. If~$U$ intersects~$X$, then let~$g$ be the inclusion map. Conversely, if such a~$g$ exists then~$(X\setminus U)\cup A$ must differ from~$X$ by the~$\epsilon$-surjection property for~$(X,A)$, so~$U$ intersects~$X$.

The finite simplicial complex~$X$ has good topological properties because it is a compact Absolute Neighborhood Retract (ANR), which means that any copy of~$X$ in~$Q$ is a retract of some neighborhood of that copy. In the detailed proof (see Sections \ref{sec_anr} and \ref{sec_proof21} in the Appendix), we show how to use these properties to prove that the existence of such a function~$g$ can be detected by an algorithm if~$(X,A)$ is semicomputable. The main idea is that one does not need to search for an arbitrary continuous function~$g$, but for a computable one. Therefore, one can test whether an open set~$U$ intersects~$X$, which makes~$X$ computable.

\subsection{The \texorpdfstring{$\epsilon$}{epsilon}-surjection property is equivalent to the local surjection property}
In this section we give an informal proof of the equivalence~$2.\Leftrightarrow 3.$ in Theorem \ref{thm_main}. The detailed argument is given in the appendix (Sections \ref{sec_app1} and \ref{sec_app2}).

\subparagraph{The \texorpdfstring{$\epsilon$}{epsilon}-surjection property implies the local surjection property.}
It is easy to see that if a local cone pair does not have the surjection property, then for any~$\epsilon>0$, the pair~$(X,A)$ does not have the~$\epsilon$-surjection property. It relies on the particular property of a cone that it contains arbitrarily small copies of itself, obtained by scaling it down: for any~$\lambda\in (0,1)$, the set~$K_i(\lambda)=\{x\in X:x_i\geq \lambda\}$ is a copy of~$K_i$ and it has arbitrarily small diameter as~$\lambda$ approaches~$1$. Given~$\epsilon>0$, consider~$\lambda$ such that~$K_i(\lambda)$ has diameter less than~$\epsilon$. Take a non-surjective function~$f$ from~$K_i(\lambda)$ to itself which is the identity on the corresponding set~$M_i(\lambda)$, and extend it to a non-surjective function~$g:X\to X$ by simply defining~$g(x)=x$ for~$x$ outside~$K_i(\lambda)$. One has~$d(g,\id_X)<\epsilon$, showing that~$(X,A)$ does not have the~$\epsilon$-surjection property.

\subparagraph{The local surjection property implies the \texorpdfstring{$\epsilon$}{epsilon}-surjection property.}
Now, assume that for every~$\epsilon>0$,~$(X,A)$ does not have the~$\epsilon$-surjection property. We show that some local cone pair does not have the surjection property. The idea is to start from a sufficiently small~$\epsilon>0$, to be defined later, and a non-surjective function~$h:X\to X$ such that~$h|_{A}=\id_A$ and~$d(h,\id_X)<\epsilon$ and consider its restriction~$h_0$ to a local cone~$K$ which is not contained in the image of~$h$. This function~$h_0$ does not immediately disprove the surjection property for the local cone pair~$(K,M)$ because~$h_0(K)$ may not be contained in~$K$ and~$h_0$ may not be the identity on~$M$. However,~$h_0$ almost satisfies these properties:~$h_0(K)$ is at distance~$\epsilon$ from~$K$ and~$h_0$ is~$\epsilon$-close to the identity on~$M$. Again, using the fact that~$K$ is a compact Absolute Neighborhood Retract (ANR) and the properties derived from that, if one takes~$\epsilon$ sufficiently small, then one can transform~$h_0$ into a continuous function~$G$ that sends~$K$ to itself, is the identity on~$M$ and is still non-surjective. Therefore,~$(K,M)$ does not have the surjection property.

\subsection{Computable type implies the \texorpdfstring{$\epsilon$}{epsilon}-surjection property}
We prove~$1.\Rightarrow 2.$ in Theorem \ref{thm_main}. We show that if a simplicial pair~$(X,A)$ does not have the~$\epsilon$-surjection property for any~$\epsilon>0$, then it has a semicomputable copy in~$Q$ that is not computable. In order to build that semicomputable copy, we show that the pair fails in a computable way to have the~$\epsilon$-surjection property, which is expressed by Definition \ref{def_witness}.

For two non-empty compact sets~$A,B\subseteq Q$, their Hausdorff distance is
\begin{equation*}
d_H(A,B)=\max(\max_{a\in A}d(a,B),\max_{b\in B}d(b,A)).
\end{equation*}

\begin{definition}\label{def_witness}
Let~$\epsilon>0$ and~$(X,A)\subseteq Q$ fail to have the~$\epsilon$-surjection property. Say that~$\delta>0$ is an~$\epsilon$\textbf{-witness} if there exists a continuous function~$f:X\to X$ such that~$f|_A=\id_A$,~$d_X(f,\id_X)<\epsilon$ and~$d_H(f(X),X)>\delta$.

Say that~$(X,A)$ has \textbf{computable witnesses} if there is a computable function~$\epsilon\mapsto\delta(\epsilon)$ such that for every~$\epsilon>0$,~$\delta(\epsilon)$ is an~$\epsilon$-witness.
\end{definition}

For a compact pair~$(X,A)$ (not necessarily simplicial), having computable witnesses is sufficient to build a semicomputable copy which is not computable.

\begin{theorem}\label{thm_witness}
Let~$(X,A)\subseteq Q$ be a computable pair having computable witnesses.~$(X,A)$ does not have computable type.
\end{theorem}
We give some intuition about the proof, and include the detailed argument in the appendix (Section \ref{sec_witness}).
\begin{proof}[Informal proof]
In order to give some intuition, let us show precisely another but related result: if we only assume that~$(X,A)$ does not have the surjection property, then one can encode the halting problem for \emph{one} program~$p$ in a copy of~$(X,A)$, in the following sense. Given~$p$, one can produce an algorithm that semicomputes a copy~$(X_p,A_p)$ of~$(X,A)$; any algorithm computing~$X_p$ could be used to decide whether~$p$ halts.

Let~$(X_0,A_0)\subseteq Q$ be a semicomputable copy of~$(X,A)$ and~$\delta>0$ be such that there exists a non-surjective continuous function~$f:X_0\to X_0$ such that~$f|_{A_0}=\id_{A_0}$ and~$d_H(X_0,f(X_0))>\delta$.

Given a program~$p$, we define a copy~$(X_p,A_p)$.  If~$p$ does not halt, then~$(X_p,A_p)=(X_0,A_0)$. If~$p$ halts, then~$(X_p,A_p)$ is another copy~$(X_1,A_1)$ defined by the following algorithm.

Start enumerating the complements of~$X_0$ and~$A_0$. If~$p$ eventually halts then consider a copy~$(X_1,A_1)$ of~$(X_0,A_0)$ with the following properties:
\begin{itemize}
\item $(X_1,A_1)$ is compatible with (i.e.~disjoint from) the current enumeration of the complements of~$X_0$ and~$A_0$,
\item $d_H(X_1,X_0)>\delta$.
\end{itemize}
The existence of~$f$ implies the existence of~$(X_1,A_1)$, which can be effectively found. We then continue enumerating the complements of~$X_1$ and~$A_1$.

We have just given an algorithm that semicomputes a copy~$(X_p,A_p)$ of~$(X,A)$, be it~$(X_0,A_0)$ or~$(X_1,A_1)$. Any algorithm that computes~$X_p$ could be used to know whether~$p$ halts:~$p$ halts if and only if~$d_H(X_p,X_0)>\delta$, which can be decided from the computable information about~$X_p$.

Now, assuming that~$(X,A)$ does not have the~$\epsilon$-surjection property for any~$\epsilon$, and using the assumption that a witness~$\delta(\epsilon)$ can be computed from any~$\epsilon$, we apply this strategy against all the programs in parallel and at infinitely many scales. The idea is simple but the details are rather technical and fully described in the appendix.
\end{proof}

Note that the standard realization of a simplicial pair is computable. We now show that if it has witnesses, then it always have \emph{computable} witnesses, which together with Theorem \ref{thm_witness} concludes the proof of $1.\Rightarrow 2.$ in Theorem \ref{thm_main}.
\begin{proposition}\label{prop_witness}
If a simplicial pair~$(X,A)$ does not have the~$\epsilon$-surjection property for any~$\epsilon>0$, then its standard realization has computable witnesses.
\end{proposition}
\begin{proof}
By~$3.\Rightarrow 2.$ in Theorem \ref{thm_main}, there exists a local cone pair~$(K_i,M_i)$ which does not have the surjection property, so there exists a non-surjective function~$f_0:K_i\to K_i$ such that~$f_0|_{M_i}=\id_{M_i}$. One can assume w.l.o.g.~that $d_X(f_0,\id_X)<1$. Let~$\delta_0>0$ be such that~$d_H(f_0(X),X)>\delta_0$. Given~$\epsilon>0$, the number~$\delta=\delta_0\epsilon$ can be computed from~$\epsilon$ and is an~$\epsilon$-witness. Indeed, the function~$f$ obtained by applying~$f_0$ to a version of~$K_i$ scaled by a factor~$\epsilon$ and extended as the identity elsewhere satisfies all the conditions.
\end{proof}


\section{\label{sec:Finite-unions-and}Techniques for the (\texorpdfstring{$\epsilon$}{epsilon}-)surjection property}
Theorem \ref{thm_main} enables one to reduce the computable type property to topological properties, namely the~$\epsilon$-surjection property and the surjection property for local cone pairs. Proving or disproving these properties may not be straightforward, so we develop a few techniques that help in many cases.

\subsection{Finite union}
The first result is a way to prove that a simplicial pair has the~$\epsilon$-surjection property by decomposing it as a finite union of pairs that all have the~$\epsilon$-surjection property. 

\begin{theorem}[Finite union]\label{thm_finiteunion}
Let~$(X,A)$ be a finite simplicial pair and let~$(X_i,A_i)_{i\leq n}$ be pairs of subcomplexes such that~$X=\bigcup_{i\leq n}X_i$ and~$A=\bigcup_{i\leq n}A_i$. If every pair~$(X_i,A_i)$ has the~$\epsilon$-surjection property for some~$\epsilon>0$, then~$(X,A)$ has the~$\delta$-surjection property for some~$\delta>0$.
\end{theorem}

We give  here the main idea and put the details in the appendix (Section \ref{sec_finiteunion}). 
\begin{proof}[Informal proof]
We are using good topological properties of finite simplicial complexes. For each~$i$, there exists a neighborhood~$U_i$ of~$X_i$ and a retraction~$r_i:U_i\to X_i$ with a special property: if~$x$ belongs to the topological interior of~$X_i$, then the only preimage of~$x$ by~$r_i$ is~$x$.

Let~$\delta$ be sufficiently small and assume that~$(X,A)$ does not have the~$\delta$-surjection property. Let~$f:X\to X$ be continuous, non-surjective and satisfy~$f|_A=\id_A$ and~$d_X(f,\id_X)<\delta$. There must be~$i\leq n$ and~$x$ in the interior of~$X_i$ that is not in the image of~$f$. We can then create a function~$f_i:X_i\to X_i$ as follows:~$f_i$ is the restriction of~$r_i\circ f$ to~$X_i$ (it is possible if~$\delta$ is sufficiently small, so that~$f(X_i)\subseteq U_i$).

The special property of~$r_i$ implies that~$x$ is not in the image of~$f_i$. Moreover,~$f_i$ is continuous, is the identity on~$A_i$ and is~$\epsilon$-close to~$\id_{X_i}$ if~$\delta$ is sufficiently small.
\end{proof}

\begin{corollary}\label{cor_finiteunion}
Let~$(X,A)$ be a simplicial pair and~$(X_i,A_i)_{i\leq n}$ be pairs of subcomplexes such that~$X=\bigcup_{i\leq n}X_i$ and~$A=\bigcup_{i\leq n}A_i$. If every pair~$(X_i,A_i)$ has computable type, then~$(X,A)$ has computable type.
\end{corollary}

For instance, if a finite simplicial complex~$X$ is a finite union of subcomplexes that are homeomorphic to spheres, then~$X$ has computable type. More generally, if a finite simplicial pair~$(X,A)$ is a finite union of pairs of subcomplexes~$(X_i,A_i)$ that are homeomorphic to pairs~$(\S_n,\emptyset)$ or~$(\B_{n+1},\S_n)$, then~$(X,A)$ has computable type.

%

\subsection{Cone of a graph}

In a~$2$-dimensional simplicial pair, the local cones are cones of graphs. We obtain a characterization of the surjection property for such cones. In order to state the result, we need to define the cone pair induced by a pair, already informally discussed in Remark \ref{rmk_conepair}. Let~$(L,N)$ be a pair. We define the cone pair~$(K,M):=\cone{L,N}$ as follows:
\begin{itemize}
\item $K=\cone{L}$ is the quotient of~$L\times [0,1]$ by the equivalence relation~$(x,0)\sim (y,0)$,
\item $M=L\cup \cone{N}$, where~$L$ is embedded in~$K$ as~$L\times\{1\}$.
\end{itemize}
The space~$L$ is called the \textbf{base} of the cone~$K=\cone{L}$, and the equivalence class~$L\times\{0\}$ is called the \textbf{tip} of~$K$.

\begin{example}
Let us illustrate this notion on the usual example of balls and spheres:
\begin{itemize}
\item $\cone{\S_n,\emptyset}=(\B_{n+1},\S_n)$ with the tip at the center of~$\B_{n+1}$,
\item $\cone{\B_n,\S_{n-1}}=(\B_{n+1},\S_n)$ with the tip in~$\S_n$.
\end{itemize}
\end{example}

Here is the main result of this section.
\begin{theorem}\label{thm_cone_graph}Let~$(L,N)$ be a pair such that~$L$ is a
finite graph and~$N$ is a subset of its vertices. The following statements are equivalent:
\begin{enumerate}
\item $\cone{L,N}$ has the surjection property,
\item Every edge is in a cycle or a path starting and ending in~$N$.
\end{enumerate}
\end{theorem}

We follow the usual convention that in a graph, a path and a cycle do not visit a vertex twice, i.e.~they are topologically a line segment and a circle respectively. In particular, a path connects two different points.

The proof is given in the appendix (Section \ref{sec_app_graph}).

\begin{example}[Star pair]\label{ex_star}
Fix some~$n\geq 1$ and let~$X$ be the star with~$n$ branches and~$A$ be the~$n$ endpoints of these branches (see Figure \ref{fig_star}), with a special case for~$n=1$:~$\cone{\{v\},\emptyset}=(\B_1,\S_0)$. The pair~$(X,A)$ is precisely~$\cone{A,\emptyset}$. As~$A$ has no edge, it satisfies the conditions of Theorem \ref{thm_cone_graph}, therefore~$(X,A)$ has the surjection property. One can then obtain Iljazovi{\'c}'s result that every finite graph has computable type \cite{2020graphs}, because the local cones of a finite graph are stars, which have the surjection property.
\end{example}
\begin{figure}[h]
\centering
\subcaptionbox{Star with $5$ branches}[4cm]{\includegraphics{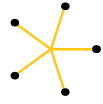}}
\subcaptionbox{Star with $1$ branch}[4cm]{\includegraphics{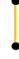}}
\caption{The star pairs~$(X,A)$ have the surjection property (Example \ref{ex_star}) ($X$ in yellow,~$A$ in black)}\label{fig_star}
\end{figure}

\begin{example}[$n$ squares]\label{ex_squares}
Fix some~$n\geq 2$ and let~$X$ be the union of~$n$ squares which
all meet in one common edge and~$A$ be the union of all the
other edges (see Figure \ref{fig_squares}). The pair~$(X,A)$ has the surjection property. Indeed,~$(X,A)=\cone{A,\emptyset}$ and~$A$ is a graph which is a union of circles (each circle is the boundary of the union of two squares). Therefore,~$\cone{A,\emptyset}$ has the surjection property by Theorem \ref{thm_cone_graph}. Finally,~$(X,A)$ has computable type by Corollary \ref{cor_cone_pair}.
\begin{figure}[h]
\centering
\includegraphics{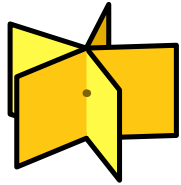}
\caption{A union of~$5$ squares is the cone of a graph; the tip is at the center, the graph is in black (Example \ref{ex_squares}).}\label{fig_squares}
\end{figure}
\end{example}
We expect a generalization of Theorem \ref{thm_cone_graph} to cones of arbitrary simplicial complexes, by using the notions of~$n$-cycles and relative~$n$-cycles from homology, generalizing cycles and paths respectively \cite{Hatcher2002}.

In the next section we apply Theorem \ref{thm_cone_graph}, giving an example of a cone pair of a graph which does not have the surjection property.

\subsection{The dunce hat}
The \textbf{dunce hat}~$D$ is the space obtained from a solid triangle by gluing
its three sides together, with the orientation of one side reversed (see Figure \ref{fig_dunce_a}). It is a classical example, introduced by Zeeman \cite{Zeeman63}, of a space that is contractible but not intuitively so. It is a 2-dimensional simplicial complex with no free edge, i.e.~no edge that belongs to one triangle only.

\begin{theorem}
\label{thm_dunce_hat} The dunce hat does not have computable type.
\end{theorem}

\begin{proof}
It is possible to turn the dunce hat into a simplicial complex. The
vertices of the triangle are identified to a point~$v$, and the local
cone pair at that point is~$\cone{L,N}$ where~$L=C_{1}\vee I\vee  C_{2}$
is the graph consisting of two circles~$C_1,C_2$ joined by a line segment~$I$, and~$N$ is empty (see Figure \ref{fig_dunce_c}).

We apply Theorem \ref{thm_cone_graph}:~$L$ is a finite graph containing an edge~$I$ which is neither in a cycle nor in a path from~$N$ to~$N$ ($N$ is empty), therefore~$\cone{L,N}$ does not have the surjection property. Theorem \ref{thm_main} then implies that the dunce hat does not have computable type.
\begin{figure}[h]
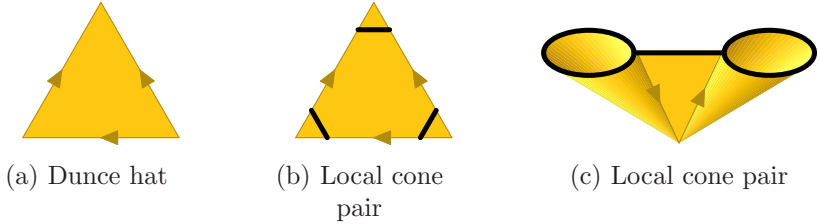

\centering
\subcaptionbox{Dunce hat\label{fig_dunce_a}}{$\quad$\includegraphics{image-0}}\hspace{1cm}
\subcaptionbox{Local cone pair\label{fig_dunce_b}}{$\quad$\includegraphics{image-100}}\hspace{1cm}
\subcaptionbox{Local cone pair\label{fig_dunce_c}}{\includegraphics{image-1}}
\caption{(a) The dunce hat is obtained by gluing the edges with the indicated orientations; (b) and (c) a local cone pair~$(\cone{L},L)=\cone{L,\emptyset}$ with tip at~$v$, with~$L$ in black.}\label{fig_dunce}
\end{figure}
\end{proof}

As far as we know, there is no simple and visual way of building a semicomputable copy of the dunce hat that is not computable, i.e.~the involved construction carried out in the proof of Theorem \ref{thm_main} cannot be avoided. The same remark applies to the pair~$(\cone{L},L)$ depicted in Figure \ref{fig_dunce_c}.

If~$A$ is the identified edges of the triangle, then it can be proved, by analyzing its local cone pairs, that the pair~$(D,A)$ has computable type. In particular, the local cone pair at~$v$ is~$\cone{L,N}$ where~$N$ consists of the two endpoints of the middle interval, so~$L$ is the union of two circles and a line segment between two points of~$N$, hence~$\cone{L,N}$ has the surjection property by Theorem \ref{thm_cone_graph}.

\begin{remark}[Quotient vs pair]
It was proved in \cite{CelarICCA21} that for any compact pair~$(X,A)$ where~$A$ has empty interior, if the quotient space~$X/A$ has computable type then the pair~$(X,A)$ has computable type. It is also proved that the converse implication fails, the counter-example is given by the circle~$X$ and a subset~$A$ consisting of a converging sequence together with its limit. The pair~$(X,A)$ has computable type, simply because~$X$ itself has computable type. However,~$X/A$ is homeomorphic to the Hawaiian earring which does not have computable type. This quotient is not a finite simplicial complex.

We give an other counter-example of a quotient space which is a finite simplicial complex. Let~$L=C_1\vee I\vee C_2$,~$X$ be the cylinder of~$L$ and~$A$ the two bases of the cylinders. Inspecting the local cones one can show that~$(X,A)$ has computable type but~$X/A$ does not.
\end{remark}

\subsection{Bing's house, or the house with two rooms}
All the known examples of sets having computable type are non-contractible (note that we are not considering pairs, but single sets), and one might conjecture that no contractible set has computable type. We give a counter-example, which is a famous space that was defined as a counter-example for other properties. It was invented by Bing \cite{Bing64} and is now called \textbf{Bing's house}, or the house with two rooms. The set is depicted in Figure \ref{fig_bing}, together with a half-cut to help visualizing it. It is an example of a space which is contractible but not intuitively so. It can be endowed with a simplicial complex structure (by triangulating each flat surface). It is then a 2-dimensional simplicial complex with no free edge, which means that every edge belongs to at least two triangles.
\begin{figure}[h]
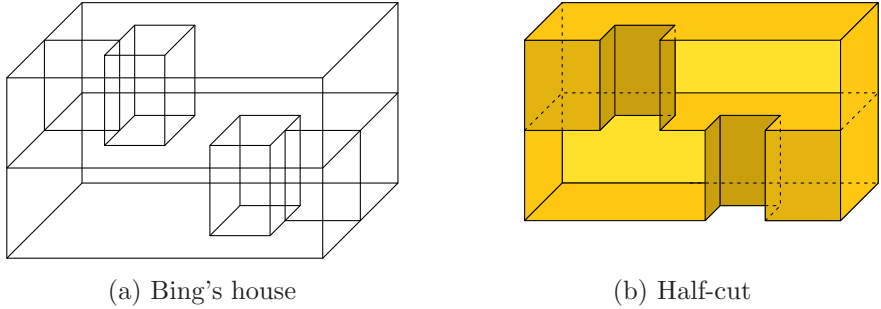

\centering
\subcaptionbox{Bing's house}{\includegraphics{image-10}}\hspace{1cm}
\subcaptionbox{Half-cut}{\includegraphics{image-11}}
\caption{Bing's house with two rooms and a half-cut of it (the full house is obtained by adding the symmetric reflection of the half-cut through the front vertical plane). It consists of two rooms, each of which can be accessed from outside through a tunnel crossing the other room. Each tunnel is linked by an internal wall to a side wall.}\label{fig_bing}
\end{figure}

Using our results we easily show that this set has computable type as a single set, i.e.~without adjoining a boundary to it.
\begin{theorem}
Bing's house has computable type.
\end{theorem}

It is worth noticing that thanks to our results, it can be proved by looking at pictures only, although the argument can be formalized.

\begin{proof}
Using Theorem \ref{thm_main}, it is sufficient to inspect the possible local cones. One easily sees that there are three types of possible cones, depicted in Figure \ref{fig_bing_cones}. The basis of each cone is a graph which is a union of 1, 2 or 3 cycles, so by Theorem \ref{thm_cone_graph} each cone pair has the surjection property, therefore Bing's house has computable type by Theorem \ref{thm_main}.
%
%

\begin{figure}[h]
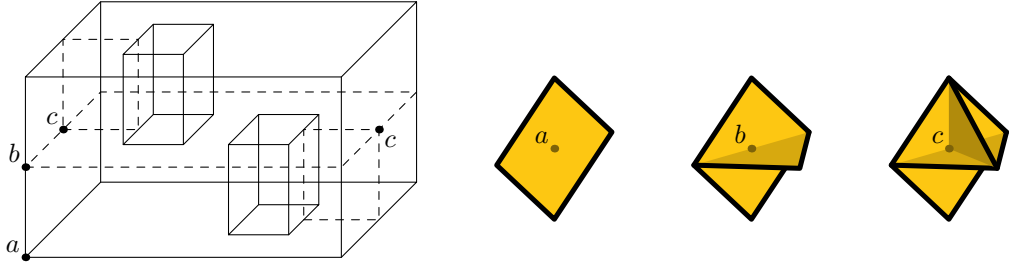

\centering
\includegraphics{image-30}\hspace{1cm}\includegraphics{image-12}\hspace{1cm}\includegraphics{image-13}\hspace{1cm}\includegraphics{image-14}
\caption{The local cones in Bing's house: their bases (in black) are graphs that are unions of cycles. Each point of Bing's house is the tip of one of these three cones: two points are tips of the third cone, all the other points on the dashed lines are tips of the second cone, all the other points are tips of the first cone.}\label{fig_bing_cones}
\end{figure}
\end{proof}

\section{Boundary}\label{sec_boundary}
Given a simplicial complex~$X$, a natural problem is to understand whether there is a minimal notion of boundary~$\partial X$ such that the pair~$(X,\partial X)$ has computable type. We make a few observations about three possible candidates. Let
\begin{itemize}
\item $\partial_1X$ be the union of simplices that are contained in \emph{exactly} one simplex of the next dimension, i.e.~$\partial_1 X$ is the union of the free simplices of~$X$,
\item $\partial_+X$ be the union of simplices that are contained in \emph{at least} one simplex of the next dimension,
\item $\oddbd{X}$ be the union of simplices that are contained in \emph{an odd number} of simplices of the next dimension.
\end{itemize}

In the proofs of the next results, we say that a simplex~$M$ in~$X$ is \textbf{maximal} if it is not contained in a higher-dimensional simplex of~$X$. 

\begin{proposition}
\label{prop:Every-simplicial-pair}Every simplicial pair~$(X,\partial_+ X)$ has computable type.
\end{proposition}
\begin{proof}
Let~$(M_i)_{i\leq n}$ be an enumeration of the maximal simplices of~$X$. $M_i$ is a ball, let~$\partial M_i$ be its bounding sphere, which is a subcomplex of~$M_i$. One has~$X=\bigcup_{i\leq n} M_i$ and~$\partial_+ X=\bigcup_{i\leq n}\partial M_i$. Each pair~$(M_i,\partial M_i)$ has the surjection property (Example \ref{example ball sphere}), so~$(X,\partial_+ X)$ has the~$\epsilon$-surjection property for some~$\epsilon$ by Theorem \ref{thm_finiteunion}. As a result,~$(X,\partial_+ X)$ has computable type by Theorem \ref{thm_main}.
\end{proof}

\begin{proposition}\label{prop_free_faces}
Let~$X$ be a finite simplicial complex and~$A$ a subcomplex. If~$(X,A)$ has computable type, then~$A$ contains~$\partial_1 X$.
\end{proposition}
\begin{proof}
Assume that some simplex~$\Delta$ belongs to~$\partial_1 X$ but not to~$A$. We show that for every~$\epsilon>0$, ~$(X,A)$ does not have the~$\epsilon$-surjection property, implying that~$(X,A)$ does not have computable type by Theorem \ref{thm_main}. Let~$\epsilon>0$. Let~$\Delta'$ be the unique maximal simplex having~$\Delta$ as a face ($\Delta'$ has one more vertex than~$\Delta$). There is a non-surjective function~$f:\Delta'\to\Delta'$ which is~$\epsilon$-close to the identity and is the identity on the other faces of~$\Delta'$:~$f$ slightly pushes points of~$\Delta'$ away from~$\Delta$. We extend~$f$ as the identity on the rest of~$X$, which gives a continuous function because~$\Delta$ is free. As~$\Delta$ is not in~$A$,~$f$ is the identity on~$A$.
\end{proof}

The following observations can be made:
\begin{itemize}
\item \label{candidat 1}Although~$(X,\partial_1 X)$ has computable type when~$X$ is a~$1$-dimensional complex (i.e., a graph), it is no more true for~$2$-dimensional complexes. For the dunce hat~$D$, one has~$\partial_1 D=\emptyset$ but we saw in Theorem \ref{thm_dunce_hat} that~$(D,\emptyset)$ does not have computable type.
\item While~$(X,\partial_+ X)$ always has computable type by Proposition \ref{prop:Every-simplicial-pair},~$\partial_+ X$ is far from optimal. For instance, it is always non-empty (unless~$X$ is a single point), but for any sphere~$\S_n$, the pair~$(\S_n,\emptyset)$ already has computable type.
\item In a subsequent paper we prove that~$(X,\oddbd{X})$ always has computable type, using homology. Observe that~$\oddbd{X}$ is in general not optimal, as the example of graphs shows:~$(X,\partial_1 X)$ has computable type and~$\partial_1 X$ is usually smaller than~$\oddbd{X}$, which contains all the vertices of odd degrees.
\end{itemize}
%

\section{Open questions and generalization}\label{sec_conclusion}
We leave two open questions.
\begin{question}
Is there a canonical notion of boundary~$\partial X$ for a simplicial complex~$X$, such that~$(X,\partial X)$ always has computable type, and~$\partial X$ is minimal in some sense?
\end{question}

\begin{question}
For simplicial pairs~$(L,N)$, is it possible to characterize the surjection property for~$\cone{L,N}$ in terms of the homology of~$(L,N)$?
\end{question}

We finally mention that the proof of the main result actually applies to more general spaces. For instance one can prove that if~$(M,\partial M)$ is a compact manifold with boundary, then~$\cone{M,\partial M}$ has computable type because it satisfies the surjection property, although it is not always a simplicial complex. These results will appear in a forthcoming article.

\bibliography{biblio}

\newpage
\appendix

\section{Proof of Theorem \ref{thm_main}}

\subsection{Absolute Neighborhood Retracts (ANRs)}\label{sec_anr}
A first property of finite simplicial complexes is that they are Absolute Neighborhood Retracts (ANRs). This important notion was introduced by Borsuk \cite{Borsuk32} and plays an eminent role in algebraic topology. Moreover, it has very useful computability-theoretic consequences, which will be used in the proof. We point out that the computability-theoretic aspects of compact ANRs has been studied by Collins in \cite{Collins09}, although we do not use these results.

\begin{definition}Let~$X$ be a compact space.
\begin{enumerate}
\item~$A\subseteq X$ is a \textbf{neighborhood retract (NR)} if it is a retract of a neighborhood of~$A$ in~$X$,
\item~$X$ is an \textbf{absolute retract (AR)} if every copy of~$X$ in~$Q$ is a retract of~$Q$,
\item~$X$ is an \textbf{absolute neighborhood retract (ANR)} if every copy of~$X$ in~$Q$ is a NR.
\end{enumerate}
\end{definition}

We recall bellow some classical facts (see \cite{1951some_theorems_on_ANR}
and \cite{2001vanMill}).

\begin{fact}\label{fact_ANR}
We have the following.
\begin{enumerate}[a.]
\item \label{fact: simplicial anr}A finite simplicial complex is an ANR,
\item \label{fact: cone ar}A cone of an ANR is an AR,
\item \label{fact: retract of ar}A retract of an AR is an AR,
\item \label{fact: disk ar}An~$n$-dimensional ball is an AR,
\item \label{fact: extension ar}If~$Y$ is an AR and~$(X,A)$ is a pair, then every continuous function~$f:A\to Y$ has a continuous extension~$F:X\to Y$,
\item \label{fact_ANR_NR}In an ANR, every NR is an ANR.
\end{enumerate}
\end{fact}

The following classical result enables one to define a continuous function piece by piece: if the pieces are consistent, then continuity automatically follows.

\begin{lemma}\label{lem_piecewise}
Let~$X,Y$ be topological spaces and~$A,B$ be closed subsets of~$X$ such that~$X=A\cup B$. If~$f:A\to Y$ and~$g:B\to Y$ are continuous and coincide on~$A\cap B$, then their common extension~$h:X\to Y$ is continuous.
\end{lemma}
\begin{proof}
Let~$V\subseteq Y$ be open. There exist two open sets~$U_A,U_B\subseteq X$ such that~$f^{-1}(V)=U_A\cap A$ and~$g^{-1}(V)=U_B\cap B$. One has~$h^{-1}(V)=(U_A\setminus B)\cup (U_B\setminus A)\cup (U_A\cap U_B)$ which is open.
\end{proof}

We will also use the following notation.

\begin{notation}\label{not_ne}
If~$X$ is a metric space,~$A\subseteq X$ and~$r>0$, let~$\Ne(A,r)=\{x\in X:d(x,A)\leq r\}$.
\end{notation}

\subsection{Proof of \texorpdfstring{$2.\Rightarrow 1.$}{2->1} in Theorem \ref{thm_main}}\label{sec_proof21}

We first need a few lemmas.
\begin{lemma}
\label{lem: good retract}If~$Y\subseteq Q$ is a compact ANR and~$\delta>0$,
then there exists an open set~$V\supseteq Y$ that is a finite union of rational balls and a retraction~$r:V\rightarrow Y$ such that~$d_{V}(r,\id_{V})<\delta$.
\begin{proof}
As~$Y$ is an ANR, there exists an open set~$W$ containing~$Y$ and a retraction~$r:W\to Y$. By compactness of~$Y$ and continuity of~$r$, if~$\epsilon>0$ is sufficiently small then~$r$ is~$\epsilon$-close to the identity on~$Y^\epsilon$. Let~$V$ be a finite union of open rational balls covering~$Y$ and contained in~$Y^\epsilon$.
\end{proof}
\end{lemma}

The next lemma is a well-known property of ANR's that can be found in \cite{1989Infinite_dimensional_topology}: if two functions to an ANR are sufficiently close to each other, then one has a continuous extension if and only if the other has. It has important computability-theoretic consequences, because arbitrary functions can be replaced by computable functions that are close enough to the original ones.

\begin{lemma}[Exercice 4.1.5 in \cite{1989Infinite_dimensional_topology}]
\label{lem: good extension}Let~$Y$ be a compact metrizable ANR. For every~$\epsilon>0$
there exists~$0<\alpha<\epsilon$ such that for every pair~$(X,A)$,
if~$f,g:A\rightarrow Y$ are continuous and such that~$d_{A}(f,g)<\alpha$
and~$f$ has a continuous extension~$F:X\rightarrow Y$, then~$g$
has a continuous extension~$G:X\rightarrow Y$ with~$d_{X}(F,G)<\epsilon$.
\end{lemma}

\begin{lemma}\label{lem_dense_sequence}
There exists a computable sequence~$(f_{j})_{j\in\N}$ of functions~$f_j:Q\to Q$ which is dense in the metric~$d_Q$.
\end{lemma}
\begin{proof}
Say that an element~$x\in Q$ is \emph{dyadic of order}~$n$ if it has the form~$x=(x_0,\ldots,x_{n-1},0,0,0,\ldots)$ where each~$x_i$ is a multiple of~$2^{-n}$. For each~$n\in\N$, the finite set of dyadic elements of order~$n$ forms a regular grid. One can then define a piecewise affine map by assigning a dyadic element to each dyadic element of order~$n$ and interpolating affinely in between. All the possible such assignments provide a dense computable sequence of functions from~$Q$ to itself.
\end{proof}

We now prove the announced implication.
\begin{proof}[Proof of~$2.\Rightarrow 1.$ in Theorem \ref{thm_main}]
Assume that~$(X,A)$ is embedded as a semicomputable pair in~$Q$ and has the~$\epsilon$-surjection
property for some~$\epsilon>0$.~$X$ can be subdivided so that each simplex has diameter less than~$\epsilon/4$ (for instance by barycentric subdivision, see \cite{Hatcher2002}). Let~$(M_i)_{1\leq i\leq n}$ be the maximal simplices of~$X$ and~$\partial M_i$ the union of the proper faces of~$M_i$. Let~$Y_i=(X\setminus M_i)\cup \partial M_i$. One has
$A\subseteq Y_{i}$ and~$Y_i$ is an ANR because it is a finite simplicial complex. 
Let~$\alpha_{i}>0$ be provided by Lemma \ref{lem: good extension} applied to~$Y_i$ and~$\frac{\epsilon}{4}$, and let~$\alpha=\min_{i}(\alpha_{i})$. Using Lemma \ref{lem: good retract},
for every~$i$, let~$V_i\supseteq Y_{i}$ be a finite union of rational balls
and~$r_{i}:V_{i}\rightarrow Y_{i}$ a retraction such that~$d_{V_{i}}(r_{i},\id_{V_{i}})<\min(\epsilon/4,\alpha/2)$.
Let~$(f_{j})_{j\in\N}$ be a dense computable sequence of functions from~$Q$ to itself provided by Lemma \ref{lem_dense_sequence}. Now, let~$U\subseteq Q$ be an open set and~$Z=(X\setminus U)\cup A$.
\begin{claim}\label{claim3}The following are equivalent:
\begin{enumerate}
\item $U$ intersects~$X$,
\item There exist~$i\leq n$ and a continuous function~$g:Z\rightarrow Y_{i}$
such that~$g|_{A}=\id_{A}$ and~$d_{Z}(g,\id_{Z})<\epsilon/4$,
\item There exist~$i\leq n$ and~$j$ such that~$f_{j}(Z)\subseteq V_{i}$,~$d_{A}(f_{j},\id_{A})<\frac{\alpha}{2}$
and~$d_{Z}(f_{j},\id_{Z})<\epsilon/2$.
\end{enumerate}
\end{claim}
\begin{claimproof}[Proof of Claim \ref{claim3}]
$1\Rightarrow2$: Let~$i\leq n$ be such that~$M_i\cap U\neq\emptyset$. Take~$x\in U\cap M_i\setminus \partial M_i$ and a continuous retraction~$r:X\setminus\{x\}\rightarrow Y_{i}$. Let~$g$ be the restriction of~$r$ to~$Z$. $g$ satisfies the conditions because~$A\subseteq Y_i$ and the diameter of~$M_i$ is less than~$\epsilon/4$.

$2\Rightarrow3$: Note that the conditions that~$f_j$ should satisfy are satisfied by~$g$ and by any function that is sufficiently close to~$g$. As the sequence~$(f_j)_{j\in\N}$ is dense, one can take~$f_j$ arbitrarily close to~$g$, so that~$f_j$ satisfies the required conditions.

$3\Rightarrow1$: Suppose that~$U$ is disjoint from~$X$, i.e.~$Z=X$. Let~$F=r_{i}\circ f_{j}:X\rightarrow Y_{i}$. One has~$d_{A}(F,\id_{A})=d_{A}(r_{i}\circ f_{j},\id_{A})\leq d_{A}(r_{i}\circ f_{j},f_{j})+d_{A}(f_{j},\id_{A})<\alpha/2+\alpha/2=\alpha$.
Therefore using Lemma \ref{lem: good extension} there exists a continuous extension
$G:X\rightarrow Y_{i}$ of~$\id_{A}$ such that~$d_{X}(G,F)<\epsilon/4$.
One has~$d_{X}(G,\id_{X})\leq d_{X}(G,r_{i}\circ f_{j})+d_{X}(r_{i}\circ f_{j},f_{j})+d_{X}(f_{j},\id_{X})<\epsilon/4+\epsilon/4+\epsilon/2=\epsilon$,
which contradicts the~$\epsilon$-surjection property of~$(X,A)$.
\end{claimproof}

If~$U$ is a rational ball in~$Q$, then 3.~is semidecidable. As 3.~is equivalent to 1., one can semidecide which rational balls~$U$ intersect~$X$, therefore~$X$ is computable.
\end{proof}

\subsection{Proof of \texorpdfstring{$2.\Rightarrow 3.$}{2->3} in Theorem \ref{thm_main}}\label{sec_app1}

Assume that there is a local cone pair~$(K,M)$ which does not have the surjection property, i.e.~there exists a non-surjective continuous function~$f:K\rightarrow K$ such that~$f|_{M}=\id_{M}$. Let~$\epsilon>0$ be arbitrary: we show that~$(X,A)$ does not have the~$\epsilon$-surjection property.

For any~$\lambda\in (0,1)$, the pair~$(K_\lambda,M_\lambda)$ defined by
\begin{align*}
K_\lambda&=\{x\in X:x_i\geq\lambda\},\\
M_\lambda&=(K_\lambda\cap A)\cup \{x\in A:x_i=\lambda\}
\end{align*}
is a copy of~$(K,M)$. If~$\lambda$ is sufficiently close to~$1$, then the diameter of~$K_\lambda$ is smaller than~$\epsilon$. The function~$f:K\to K$ can be translated to a function~$f_\lambda:K_\lambda\to K_\lambda$ (let~$\phi_\lambda:K\to K_\lambda$ be a homeomorphism sending~$(K,M)$ to~$(K_\lambda,M_\lambda)$ and define~$f_\lambda=\phi_\lambda\circ f\circ \phi_\lambda^{-1}$).

We extend~$f_\lambda$ to~$g:X\to X$ by defining~$g(x)=x$ for~$x$ outside~$K_\lambda$. As~$f$ is the identity on~$M_\lambda$ which contains the topological boundary of~$K_\lambda$, the function~$g$ is continuous by Lemma \ref{lem_piecewise}. By choice of~$\lambda$,~$g$ is~$\epsilon$-close to the identity. As the part of~$A$ in~$K_\lambda$ is contained in~$M_\lambda$,~$g|_A=\id_A$. Therefore,~$g$ shows that~$(X,A)$ does not have the~$\epsilon$-surjection property.

\subsection{Proof of \texorpdfstring{$3.\Rightarrow 2.$}{3->2} in Theorem \ref{thm_main}}\label{sec_app2}
Let~$(X,A)\subseteq\R^n$ be the standard realization of a finite simplicial pair, endowed with the metric~$d(x,y)=\max_{i\leq n}|x_i-y_i|$. We assume that for every~$\epsilon>0$,~$(X,A)$ does not have the~$\epsilon$-surjection property, and we prove that some local cone pair~$(K_i,M_i)$ does not have the~$\epsilon$-surjection property.

Note again that for~$\lambda\in (0,1)$, the pair~$(K_i(\lambda),M_i(\lambda))$ defined by
\begin{align*}
K_i(\lambda)&=\{x\in X:x_i\geq\lambda\},\\
M_i(\lambda)&=\{x\in A:x_i=\lambda\}\cup (K_i(\lambda)\cap A)
\end{align*}
is a copy of~$(K_i,M_i)$.  Let~$\lambda_0=1/(n+1)$ and observe that for every~$x\in X$,~$x_i\geq \lambda_0$ for some~$i\leq n$. Therefore,~$X$ is covered by the cones~$K_i(\lambda_0)$,~$i\leq n$.

Let~$\epsilon$ be small so that~$\lambda_0-2\epsilon>0$. Let~$\lambda_1=\lambda_0-\epsilon$ and~$\lambda_2=\lambda_1-\epsilon>0$. Observe that, using Notation \ref{not_ne},
\begin{align}
\Ne(K_i(\lambda_{0}),\epsilon)&\subseteq K_i(\lambda_{1})\label{eq1},\\
\Ne(K_i(\lambda_{1}),\epsilon)&\subseteq K_i(\lambda_2).\label{eq2}
\end{align}

\begin{figure}[h]
\centering
\includegraphics{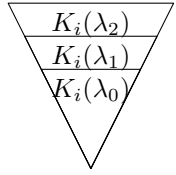}
\caption{The cones~$K_i(\lambda_0)\subseteq K_i(\lambda_1)\subseteq K_i(\lambda_2)$}
\end{figure}

Let~$\alpha<\epsilon$ be smaller than the values provided by Lemma \ref{lem: good extension} applied to the compact ANRs~$K_i(\lambda_2)$ ($i\leq n$) and~$\epsilon$. By assumption,~$(X,A)$ does not have the~$\alpha$-surjection property, i.e.~there exists a non-surjective continuous function~$h:X\rightarrow X$ such that~$h|_{A}=\id_{A}$ and~$d_{X}(h,\id_{X})<\alpha<\epsilon$. As~$X=\bigcup_iK_i(\lambda_0)$, there exists~$i\leq n$ such that
\begin{equation}
K_i(\lambda_{0})\nsubseteq h(X).\label{eq3}
\end{equation}
Let~$(K,M)=(K_i(\lambda_2),M_i(\lambda_2))$.

We now define a non-surjective continuous function~$G:K\rightarrow K$ such that~$G|_{M}=\id_{M}$, showing that~$(K,M)$ does not have the surjection property.

First observe that~$h(K_i(\lambda_1))$ is contained in~$K$. Indeed,~$h$ is~$\epsilon$-close to the identity so~$h(K_i({\lambda_1}))\subseteq \Ne(K_i(\lambda_{1}),\epsilon)\subseteq K$ by \eqref{eq2}.

We define~$g:K_i(\lambda_{1})\cup M\to K$ by
\begin{equation*}
g|_{K_i(\lambda_{1})}=h|_{K_i(\lambda_{1})}\text{ and }g|_{M}=\id_{M}.
\end{equation*}
The function~$g$ is well-defined and continuous because~$h$ and the identity coincide on~$K_i(\lambda_{1})\cap M\subseteq A$.

We now define a continuous extension~$G:K\to K$ of~$g$ using Lemma \ref{lem: good extension}. Note that~$g$ is~$\alpha$-close to the inclusion~$f:K_i(\lambda_1)\cup M\to K$ and~$f$ has a continuous extension~$F:=\id_K:K\to K$, so using Lemma \ref{lem: good extension},~$g$ has a continuous extension~$G:K\rightarrow K$ satisfying~$d_{K}(G,\id_{K})<\epsilon$. As~$G$ extends~$g$,~$G|_M=\id_M$. We show that~$G$ is not surjective, implying that~$(K,M)$ does not have the surjection property. Indeed,~$K_i(\lambda_0)$ is not contained in~$G(K)$:
\begin{itemize}
\item $G(K_i(\lambda_1))=h(K_i(\lambda_1))$ which does not contain~$K_i(\lambda_0)$ by \eqref{eq3},
\item $G(K\setminus K_i(\lambda_{1}))\subseteq \Ne(K\setminus K_i(\lambda_{1}),\epsilon)$ which is disjoint from~$K_i(\lambda_0)$ by \eqref{eq1}.
\end{itemize}

\section{Proof of Theorem \ref{thm_witness}}\label{sec_witness}
Let us recall the statement of Theorem \ref{thm_witness}.

\begin{theorem*}
Let~$(X,A)\subseteq Q$ be a computable pair having computable witnesses.~$(X,A)$ does not have computable type.
\end{theorem*}

We first need some background. We define a quasi-metric~$\rho$ on the space of non-empty compact subsets of~$Q$ (a quasi-metric is like a metric without the symmetry axiom):
\begin{equation*}
\rho(A,B)=\max_{b\in B}d(b,A).
\end{equation*}
Observe that~$\rho(A,B)=0\iff B\subseteq A$, and~$\rho(A,B)$ is small if~$B$ is contained in a small neighborhood of~$A$. $B$ is semicomputable iff for finite sets~$A$ and rational numbers~$r>0$, the inequality~$\rho(A,B)<r$ is semidecidable.  Note that~$d_H(A,B)=\max(\rho(A,B),\rho(B,A))$. 

The space of self-homeomorphisms of~$Q$ is a computable Polish space, i.e.~it can be endowed with a computable complete metric~$D$ (the proof of the classical result that it is a Polish space is easily effective \cite{2001vanMill}). We can assume that~$D(f,g)\geq d(f,g)$ for functions~$f,g$, replacing~$D$ by~$\max(D,d)$ if necessary. As~$D$ is computable, given a computable homeomorphism~$f$ and~$\epsilon>0$, one can compute~$\varphi(f,\epsilon)$ such that if~$g$ is a homeomorphism such that~$d(f,g)<\varphi(f,\epsilon)$, then~$D(f,g)<\epsilon$ (strictly speaking,~$\varphi(f,\epsilon)$ is not a function of~$f$ but of the representation of~$f$, but we abuse the notation for simplicity).

We assume that~$(X,A)\subseteq Q$ is a computable pair having computable witnesses and we build a homeomorphism~$f:Q\to Q$ such that~$(f(X),f(A))$ is semicomputable but~$f(X)$ is not computable. The idea is to encode a non-computable c.e.~set in~$f(X)$.

The next result is central in our construction: it enables, at any stage of the algorithm, to switch from the current copy~$h(X)$ of~$X$ to a copy~$g_1(X)$ which is almost contained in~$h(X)$ but is far away in the Hausdorff metric. In all the proof, functions from~$Q$ to~$Q$ are always homeomorphisms.

\begin{lemma}\label{lem_iter}
Given~$g_0:Q\to Q$ and~$\epsilon>0$, one can compute~$\epsilon'>0$ such that for every~$h\in \cB_D(g_0,\epsilon')$, there exists~$g_1$ satisfying:
\begin{itemize}
\item $D(g_0,g_1)<\epsilon/2$,
\item $d_H(g_0(X),g_1(X))>2\epsilon'$,
\item $\rho(h(X),g_1(X))$ is as small as we want.
\end{itemize}
\end{lemma}
\begin{proof}
Compute~$\alpha$ such that for all~$h:Q\to Q$,
\begin{equation}\label{eq_alpha}
\text{If }D(h,g_0)<\alpha \text{ and }D(g,1)<\alpha,\text{ then }D(h\circ g,g_0)<\epsilon/2.
\end{equation}
Compute~$\beta$, an~$\alpha$-witness for~$X$.  Using the modulus of uniform continuity of~$g_0^{-1}$, compute~$\epsilon'\leq \alpha$ such that for all non-empty compact sets~$Y,Z\subseteq Q$,
\begin{equation}\label{eq_epsilonprime}
\text{If }d_H(Y,Z)>\beta, \text{ then }d_H(g_0(Y),g_0(Z))>3\epsilon'.
\end{equation}

Let~$h\in \cB_D(g_0,\epsilon')$. As~$\beta$ is an~$\alpha$-witness, there exists~$g\in B_D(1,\alpha)$ such that $d_H(X,g(X))>\beta$ and~$\rho(X,g(X))$ is arbitrary small. Let~$g_1=h\circ g$. One has~$D(g_1,g_0)<\epsilon/2$ by choice of~$\alpha$, i.e.~using \eqref{eq_alpha} and~$\rho(h(X),g_1(X))$ is arbitrarily small. One has
\begin{align*}
d_H(g_0(X),g_1(X))&\geq d_H(g_0(X),g_0\circ g(X))-d_H(g_0\circ g(X),h\circ g(X))\\
&>3\epsilon'-D(g_0,h)> 2\epsilon'\quad\text{ by \eqref{eq_epsilonprime}}.\qedhere
\end{align*}
\end{proof}

\begin{claim}\label{claim_g0g1}
Assume that~$g_0,g_1,\epsilon'$ satisfy the conditions in Lemma \ref{lem_iter}. For all~$g:Q\to Q$,
\begin{itemize}
\item If~$g\in B_D(g_0,\epsilon')$, then~$d_H(g_0(X),g(X))<\epsilon'$,
\item If~$g\in B_D(g_1,\epsilon')$, then~$d_H(g_0(X),g(X))>\epsilon'$.
\end{itemize}
\end{claim}
\begin{claimproof}
The first item is straightforward:
\begin{equation*}
d_H(g_0(X),g(X))\leq d(g_0,g)\leq D(g_0,g).
\end{equation*}

The second item holds because
\begin{align*}
d_H(g_0(X),g(X))&\geq d_H(g_0(X),g_1(X))-d_H(g_1(X),g(X))\\
&>2\epsilon'-D(g_1,g)>\epsilon'.\qedhere
\end{align*}
\end{claimproof}

For~$g_0:Q\to Q$ and~$\epsilon>0$, let~$\epsilon'(g_0,\epsilon)$ be provided by Lemma \ref{lem_iter} applied to~$g_0$ and~$\epsilon$.

\subparagraph{Construction.}
We are going to define a homeomorphism~$f$ as follows. We define a sequence of balls~$B_n=B_D(f_n,\epsilon_n)$ such that
\begin{itemize}
\item $\epsilon_n<2^{-n}$,
\item $\cB_{n+1}\subseteq B_n$.
\end{itemize}
It implies that~$(f_n)_{n\in\N}$ is a~$D$-Cauchy sequence,~$f$ is then defined as the limit of~$f_n$. The sequence~$(B_n)_{n\in\N}$ will not be computable (otherwise~$f$ and~$f(X)$ would be computable), but will be obtained as a limit.

For each~$s\in\N$, we define a computable sequence~$B_n[s]=B_D(f_n[s],\epsilon_n[s])$ satisfying the same properties as~$(B_n)_{n\in\N}$, i.e.
\begin{itemize}
\item $\epsilon_n[s]<2^{-n}$,
\item $\cB_{n+1}[s]\subseteq B_n[s]$,
\end{itemize}
and such that for each~$n$,~$B_n[s]$ does not change for sufficiently large~$s$. $B_n$ is then defined as the limit value of~$B_n[s]$. For each~$s$, we define~$f[s]$ as the limit of~$f_n[s]$. The sequence~$(f[s])_{s\in\N}$ is computable and~$f$ is the limit of~$f[s]$.

In order to define~$f_n[s]$ and~$\epsilon_s[s]$ for all~$n,s$, we fix a non-computable c.e.~set~$E$ such as the halting set, together with a computable enumeration~$n_0,n_1,\ldots$ of~$E$ without repetition. For each~$s$, the sequence~$(B_n[s+1])_{n\in\N}$ is obtained from the sequence~$(B_n[s])_{n\in\N}$ by changing it for~$n>n_s$ only. It implies that for each~$n$,~$B_n[s]$ does not change for sufficiently large~$s$. In order to make~$f(X)$ non-computable, the idea is that~$f(X)$ is close to~$f[s](X)$ if and only if only large numbers will appear in~$E$ after stage~$s$, so computing~$f(X)$ would enable to compute~$E$.

For~$s=0$, let~$f_n[0]=f[0]=\id$ for all~$n$,~$\epsilon_0[0]=1$ and inductively~$\epsilon_{n+1}[0]\leq\epsilon'(\id,\epsilon_n[0])$.

For the induction, let~$s\in\N$ and assume that all~$f_n[t]$ and~$\epsilon_n[t]$ have been defined for~$t\leq s$. Let~$n=n_{s}$. Using Lemma \ref{lem_iter}, we choose a computable function~$g_1$ satisfying:
\begin{itemize}
\item $D(f_n[s],g_1)<\epsilon_n[s]/2$,
\item $d_H(f_n[s](X),g_1(X))>2\epsilon_{n+1}[s]$,
\item $\rho(f[s](X),g_1(X))<2^{-s}$,
\end{itemize}
which is possible if~$\epsilon_{n+1}[s]\leq \epsilon'(f_n[s],\epsilon_n[s])$ and~$D(f[s],f_n[s])<\epsilon_{n+1}[s]$. As the three inequalities are semidecidable, we can effectively find~$g_1$ by exhaustive search in a dense computable sequence of self-homeomorphisms of~$Q$.

We define the sequence~$(B_p[s+1])_{p\in\N}$ as follows:
\begin{itemize}
\item For~$p\leq n$, $f_p[s+1]=f_p[s]$ and~$\epsilon_p[s+1]=\epsilon_p[s]$,
\item For~$p=n+1$,~$f_p[s+1]=g_1$ and~$\epsilon_p[s+1]=\epsilon_p[s]$,
\item For~$p\geq n+1$,~$f_{p+1}[s+1]=g_1$ and~$\epsilon_{p+1}[s+1]\leq\epsilon'(g_1,\epsilon_p[s+1])$.
\end{itemize}

Note that~$f[s+1]=g_1$. Now that the construction is complete, let us check that it satisfies the sought properties.
\subparagraph{Verification.}
\begin{claim}
$f(X)$ is semicomputable.
\end{claim}
\begin{proof}
As the sequence~$(f[s])_{s\in\N}$ is computable, the set~$f[s](X)$ is computable uniformly in~$s$. Moreover, by construction one has for every~$s$,~$\rho(f[s](X),f[s+1](X))<2^{-s}$. Therefore, for a finite set~$A$ and~$r>0$,
\begin{equation*}
\rho(A,f(X))<r\iff \exists s, \rho(A,f[s](X))+2^{-s+1}<r
\end{equation*}
which is semidecidable because~$f[s](X)$ is computable uniformly in~$s$.
\end{proof}

We now prove that~$f(X)$ is not computable. Let~$E_s=\{n_s,n_{s+1},\ldots\}$. For all~$n\leq \min E_s$, one has~$\epsilon_{n+1}[s]=\epsilon_{n+1}$.

\begin{claim}The following holds:
\begin{itemize}
\item If~$n<\min E_s$, then~$d_H(f(X),f_{n+1}[s](X))<\epsilon_{n+1}$,
\item If~$n=\min E_s$, then~$d_H(f(X),f_{n+1}[s](X))>\epsilon_{n+1}$.
\end{itemize}
\end{claim}
\begin{claimproof}
If~$n< \min E_s$, then~$B_{n+1}[t]=B_{n+1}[s]$ for all~$t\geq s$, so~$f\in B_{n+1}=B_{n+1}[s]$. In other words,~$D(f,f_{n+1}[s])<\epsilon_{n+1}[s]$, which implies~$d_H(f(X),f_{n+1}[s](X))<\epsilon_{n+1}[s]=\epsilon_{n+1}$.

If~$n=\min E_s$, then let~$t\geq s$ be such that~$n_t=n$. One has~$f\in B_{n+1}=B_{n+1}[t+1]$ and~$f_n[t]=f_n[s]=f_{n+1}[s]$,~$\epsilon_{n+1}[t+1]=\epsilon_{n+1}[s]$ so~$d_H(f(X),f_{n+1}[s](X))>\epsilon_{n+1}[s]$ by Claim \ref{claim_g0g1} applied to~$g_0=f_n[t]$, $g_1=f_{n+1}[t+1]$, $g=f$ and~$\epsilon'=\epsilon_{n+1}[t+1]$.
\end{claimproof}
As a result,~$E$ is computable relative to~$f(X)$, because if~$n\leq\min E_s$, then one can decide whether~$n\in E_s$ by comparing~$d_H(f(X),f_{n+1}[s](X))$ with~$\epsilon_{n+1}[s]=\epsilon_{n+1}$.

\begin{lemma}
$E$ is computable if given~$n,s$ such that~$n\leq \min E_s$, one can decide whether~$n\in E_s$.
\end{lemma}
\begin{proof}
One can compute~$\min E_s$ by starting with~$n=0$, deciding whether~$n\in E_s$ and increment~$n$ until~$n\in E_s$. Therefore, given~$n$, compute~$s$ such that~$n<\min E_s$, and test whether~$n\in\{n_0,\ldots,n_s\}$.
\end{proof}
%
%
%
%

\section{Proof of Theorem \ref{thm_finiteunion}}\label{sec_finiteunion}

For each~$i$,~$X_i$ is an ANR so there exists a neighborhood~$U_i$ of~$X_i$ and a retraction~$r_i:U_i\to X_i$. We can choose~$r_i$ to have a special property.
\begin{claim}
There exists a retraction~$r_i:U_i\to X_i$ such that if~$x$ belongs to the interior of~$X_i$, then the only preimage of~$x$ by~$r_i$ is~$x$.
\end{claim}
\begin{claimproof}
Write~$\interior{X_i}$ for the interior of~$X_i$ in~$X$. The set~$X_i\setminus \interior{X_i}$, which is the topological boundary of~$X_i$ in~$X$, is a subcomplex of~$X_i$ so it is an ANR. Therefore, there exists a retraction~$s_i$ from a neighborhood~$V_i$ of~$X_i\setminus \interior{X_i}$ to that set. We then define~$U_i=V_i\cup X_i=V_i\cup\interior{X_i}$ and~$r_i$ as follows:
\begin{align*}
r_i&=s_i\text{ on }V_i\setminus\interior{X_i},\\
r_i&=\id\text{ on }X_i.
\end{align*}
$U_i$ is a neighborhood of~$X_i$,~$r_i$ is well-defined because~$s_i$ coincides with the identity on~$(V_i\setminus \interior{X_i})\cap X_i\subseteq X_i\setminus\interior{X_i}$. The two sets~$V_i\setminus \interior{X_i}$ and~$X_i$ are closed subspaces of~$U_i$, so~$r_i$ is continuous by Lemma \ref{lem_piecewise}.
\end{claimproof}

Let~$\delta<\epsilon/2$ be such that for each~$i\leq n$,~$\Ne(X_i,\delta)\subseteq U_i$ and~$r_i$ is~$\epsilon/2$-close to the identity on~$\Ne(X_i,\delta)$ (see Notation \ref{not_ne}). Let~$f:X\to X$ satisfy~$f|_A=\id_A$ and~$d(f,\id_X)<\delta$. Assume that~$f$ is not surjective. There exists~$i$ and~$x\in \interior{X_i}$ which is not in the image of~$f$.

One has~$f(X_i)\subseteq\Ne(X_i,\delta)\subseteq U_i$, so~$f_i:=r_i\circ f:X_i\to X_i$ is well-defined. Observe that~$x$ is not in the image of~$f_i$, because its only preimage by~$r_i$ is~$x$, which is not in the image of~$f$, so~$f_i$ is not surjective.

Both~$f$ and~$r_i$ are the identity on~$A_i$, so~$f_i$ is also the identity on~$A_i$. Finally,~$d_{X_i}(f_i,\id_{X_i})\leq d_{X_i}(r_i\circ f,f)+d_{X_i}(f,\id_{X_i})\leq \epsilon/2+\epsilon/2=\epsilon$. Therefore,~$f_i$ contradicts the~$\epsilon$-surjection property for~$(X_i,A_i)$. It implies that~$f$ is surjective. As a result,~$(X,A)$ has the~$\delta$-surjection property.

\section{Proof of Theorem \ref{thm_cone_graph}}\label{sec_app_graph}

We start by the next result.

If~$X,Y$ are two spaces with distinguished points~$x_0\in X$ and~$y_0\in Y$, their \emph{wedge sum}~$X\vee Y$ is the space obtained by attaching~$X$ and~$Y$ at~$x_0$ and~$y_0$ and identifying these two points. More formally~$X\vee Y$ is the quotient of the disjoint union~$X\sqcup Y$ by the equivalence relation~$x_0\sim y_0$.
\begin{proposition}
\label{prop_surjection_segment}Let~$L_{0}$ and~$L_{1}$ be
two compact ANRs,~$l_{0}\in L_{0}$,~$l_{1}\in L_{1}$ and
$I=[0,1]$ the line segment. Let~$(L,N)$ be a pair such that~$L=L_{0}\vee I\vee L_{1}$
(identifying~$l_{0}$ with~$0$ and~$l_{1}$ with~$1$) and~$N\subseteq L_{1}$. The cone pair~$\cone{L,N}$
does not have the surjection property.
\end{proposition}

\begin{figure}[h]
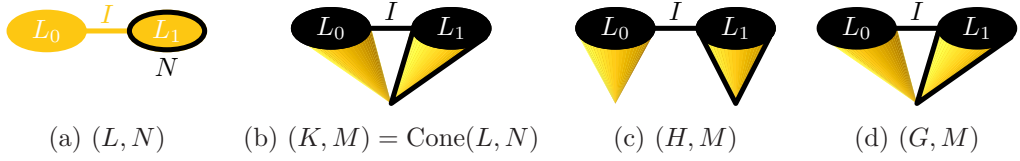

\centering
\subcaptionbox{$(L,N)$\label{fig_LN}}{\includegraphics{image-9}}
\subcaptionbox{$(K,M)=\cone{L,N}$\label{fig_cone_L}}[.28\linewidth]{\includegraphics{image-6}}
\subcaptionbox{$(H,M)$\label{fig_H}}{\includegraphics{image-7}}\hspace{.4cm}
\subcaptionbox{$(G,M)$\label{fig_G}}{\includegraphics{image-8}}
\caption{Illustration of the proof of Proposition \ref{prop_surjection_segment}. In a pair~$(X,A)$,~$X$ is yellow and~$A$ is black.}\label{fig_prop_cone}
\end{figure}

\begin{proof}
We include Figure \ref{fig_prop_cone} to help understanding the proof. Let~$(K,M)=\cone{L,N}$, i.e.~$K=\cone{L}$ and~$M=L\cup\cone{N}$. Let~$H=\cone{L_{0}}\vee I\vee\cone{L_{1}}$. We first show that~$H$ is an AR. $\cone{L_0}$ and~$\cone{L_1}$ can be embedded in~$Q_0:=[0,1/3]\times Q$ and~$Q_1:=[2/3,1]\times Q$ so that~$l_0,l_1$ are sent to~$(1/3,0,0,\ldots)$ and~$(2/3,0,0,\ldots)$, and~$I$ is embedded as~$[1/3,2/3]\times \{(0,0,\ldots)\}$. We obtain an embedding of~$H$ in~$H':=Q_0\cup I\cup Q_1$. Each~$\cone{L_i}$ is an AR, so it is a retract of~$Q_i$, hence~$H$ is a retract of~$H'$. Finally, it is not hard to see that~$H'$ is an AR, hence~$H$ is an AR.

One can see~$M$ as a subset of~$L_0\cup I\cup\cone{L_{1}}\subseteq H$ (this is the place where we use the assumption that~$N$ is contained in~$L_1$). The function~$\id_{M}:M\rightarrow H$
has a continuous extension~$F:K\rightarrow H$ because~$H$ is an
AR. Let~$G$ be the quotient of~$H$ obtained by identifying the tips of~$\cone{L_{0}}$ and~$\cone{L_{1}}$.~$G$ is a proper subset of~$K$. Hence by composing~$F$ with the quotient
map, we get a non-surjective continuous extension of~$\id_{M}$ to
a function from~$K$ to itself. Therefore,~$(K,M)$ does not have the surjection property.
\end{proof}

\begin{proof}[Proof of Theorem \ref{thm_cone_graph}]
$1.\Rightarrow 2.$ Suppose that some edge~$e=(v,w)$ is not in a cycle or a path from~$N$ to~$N$. Let~$L'$ be the graph obtained by removing~$e$ (but still containing its endpoints~$v,w$). As~$e$ is not in a cycle of~$L$,~$v$ and~$w$ belong to two different connected components of~$L'$. As~$e$ is not in a path from~$N$ to~$N$, one of these two components is disjoint from~$N$. Let~$C$ be that connected component and~$D$ be the rest of~$L'$.~$L'$ is the disjoint union of~$C$ and~$D$. Note that~$L=C\vee I\vee D$ so we can apply Proposition \ref{prop_surjection_segment} to~$L_0=C$,~$L_1=D$ and~$I=e$. It implies that~$\cone{L,N}$ does not have the surjection property.

$2.\Rightarrow 1.$ If every edge belongs to a cycle or a path from~$N$ to~$N$, then~$L$ is a union of circles, line segments with endpoints in~$N$ and isolated points. The cone of each pair~$(\S_1,\emptyset)$,~$(\B_1,\S_0)$ and~$(\B_0,\emptyset)$ has the surjection property, so~$\cone{L,N}$ has the surjection property by Theorem \ref{thm_finiteunion} (for cone pairs, the surjection property and the~$\epsilon$-surjection property are equivalent by Corollary \ref{cor_cone_pair}).
\end{proof}
\end{document}